\documentclass[12pt]{article}

\usepackage{amsmath, amsthm}
\usepackage{amsfonts,amssymb}
\usepackage{hyperref}
\usepackage[all]{xy}
\usepackage{graphicx}
\usepackage{caption}
\usepackage[usenames,dvipsnames]{pstricks}
\usepackage{bbm}
\usepackage[affil-it]{authblk}

\newcommand{\C}{\mathbb{C}}
\newcommand{\E}{\mathbb{E}}
\newcommand{\F}{\mathbb{F}}

\newcommand{\N}{\mathbb{N}}

\newcommand{\Q}{\mathbb{Q}}

\newcommand{\T}{\mathbb{T}}
\newcommand{\Z}{\mathbb{Z}}

\newcommand{\Div}{\operatorname{Div}}

\newcommand{\Pic}{\operatorname{Pic}}

\newcommand{\Gal}{\operatorname{Gal}}

\renewcommand{\ker}{\operatorname{Ker}}

\newcommand{\ord}{\operatorname{ord}}

\newcommand{\Tr}{\operatorname{Tr}}

\newcommand{\Jac}{\operatorname{Jac}}

\newcommand{\GLFl}{\operatorname{GL}_2(\F_\ell)}
\newcommand{\PGLFl}{\operatorname{PGL}_2(\F_\ell)}
\newcommand{\SLFl}{\operatorname{SL}_2(\F_\ell)}

\newcommand{\charf}{\mathbbm{1}}

\newtheorem{thm}{Theorem}

\newtheorem{lem}[thm]{Lemma}

\theoremstyle{definition}

\theoremstyle{definition}

\theoremstyle{definition}

\entrymodifiers={!!<0pt,0.7ex>+}

\title{Computing modular Galois representations}
\author{Nicolas Mascot\thanks{IMB, Universit\'e Bordeaux 1, UMR 5251, F-33400 Talence, France.
CNRS, IMB, UMR 5251, F-33400 Talence, France.
INRIA, project LFANT, F-33400 Talence, France.}
\\ \href{mailto:nmascot@math.u-bordeaux1.fr}{nmascot@math.u-bordeaux1.fr}}

\begin{document}

\maketitle

\begin{abstract}
We compute modular Galois representations associated with a newform $f$, and study the related problem of computing the coefficients of $f$ modulo a small prime $\ell$. To this end, we design a practical variant of the complex approximations method presented in \cite{EC}. Its efficiency stems from several new ingredients. For instance, we use fast exponentiation in the modular jacobian instead of analytic continuation, which greatly reduces the need to compute abelian integrals, since most of the computation handles divisors. Also, we introduce an efficient way to compute arithmetically well-behaved functions on jacobians, a method to expand cuspforms in quasi-linear time, and a trick making the computation of the image of a Frobenius element by a modular Galois representation more effective. We illustrate our method on the newforms $\Delta$ and $E_4 \cdot \Delta$, and manage to compute for the first time the associated faithful representations modulo $\ell$ and the values modulo $\ell$ of Ramanujan's $\tau$ function at huge primes for $\ell \in \{ 11,13,17,19,29\}$. In particular, we get rid of the sign ambiguity stemming from the use of a non-faithful projective representation as in \cite{Bos}. As a consequence, we can compute the values of $\tau(p) \bmod 2^{11} \cdot 3^6 \cdot 5^3 \cdot 7 \cdot 11 \cdot 13 \cdot 17 \cdot 19 \cdot 23 \cdot 29 \cdot 691 \approx 2.8 \cdot 10^{19}$ for huge primes $p$. These representations lie in the jacobian of modular curves of genus up to $22$.
\end{abstract}

\newpage

\renewcommand{\abstractname}{Acknowledgements}
\begin{abstract}
I would like to heartily thank my advisor J.-M. Couveignes for offering me this beautiful subject to work on. More generally, I would like to thank people from the Bordeaux 1 university's IMB for their support, with special thoughts to B. Allombert, K. Belabas, H. Cohen and A. Enge, as well as the PlaFRIM team. Finally, I thank B.Edixhoven for his remarks on earlier versions of this article, A. Page for helping me to make explicit the similarity classes in $\GLFl$, and T. Selig for proofreading my English.

This research was supported by the French ANR-12-BS01-0010-01 through the project PEACE, and by the DGA ma{\^\i}trise de l'information.
\end{abstract}

\newpage


\section{Introduction}

Consider a non-CM newform $f = q + \sum_{n \geqslant 2} a_n q^n \in S_k\big(\Gamma_1(N)\big)$ of weight $k \in  \N_{\geqslant 2}$, level $N \in \N^*$, and nebentypus $\varepsilon$. Denote by $K_f = \Q(a_n, n\geqslant 2)$ the number field spanned by its $q$-expansion coefficients. Let $\mathfrak{l}$ be one of its finite primes, lying over some rational prime $\ell \in \N$, let $K_{f,\mathfrak{l}}$ be the corresponding completion, and let $\Z_{K_{f,\mathfrak{l}}}$ be its ring of integers. Thanks to P. Deligne \cite{Del}, we know that there exists a continuous $\ell$-adic Galois representation
\[ G_\Q \longrightarrow \operatorname{GL}_2(\Z_{K_{f,\mathfrak{l}}}) \]
of the absolute Galois group $G_\Q = \Gal(\overline \Q / \Q)$ of $\Q$, which is unramified outside $\ell N$, and such that for all rational primes $p \nmid \ell N$, the image of the Frobenius element corresponding to any prime lying above $p$ has characteristic polynomial
\[ X^2 - a_p X + \varepsilon(p) p^{k-1} \in \Z_{K_{f,\mathfrak{l}}}[X]. \]

Assume now that $\mathfrak{l}$ has inertia degree $1$. By reducing modulo $\mathfrak{l}$, we get a mod $\ell$ representation 
\[ \rho_{f,\mathfrak{l}} \colon G_\Q \longrightarrow \GLFl. \]
\hypertarget{BadPrimes}By theorem $2.1$ of \cite{RibSL2} and lemma 2 of \cite{SwD}, for almost every $\mathfrak{l}$, the image of this representation contains $\SLFl$, and in particular this representation is irreducible. We will exclude the finitely many $\mathfrak{l}$ for which this fails to hold from now on. For instance, if we choose $f = \Delta$, we exclude $\ell = 23$, amongst others.

Further assume now that $\ell \geqslant k+1$ and that $\ell \nmid N$. In this case, this mod $\ell$ representation can then be constructed in a more concrete way as follows. Being an eigenform, $f$ has a system of Hecke eigenvalues $\lambda_f \colon \T_{k,N} \longrightarrow \Z_{K_f}$ such that
\[ Tf = \lambda_f(T) f  \qquad \forall \, T  \in \T_{k,N}, \]
where $\T_{k,N} = \Z[T_n, \, n \geqslant 2]$ denotes the Hecke algebra acting on cuspforms of weight $k$ and level $N$, and where $\Z_{K_f}$ is the ring of integers of $K_f$. Reducing modulo $\mathfrak{l}$, we get a ring morphism $\lambda_{f,\mathfrak{l}} \colon \T_{k,N} \longrightarrow \F_{\ell}$. By a weight-lowering theorem (cf \cite{Gross}, proposition 9.3 part 2), there exists another ring morphism $\mu_{f,\mathfrak{l}} \colon \T_{2,\ell N} \longrightarrow \F_{\ell}$ such that $\lambda_{f,\mathfrak{l}}(T_p) = \mu_{f,\mathfrak{l}}(T_p) \in \F_\ell$ for all rational primes $p$. This other Hecke algebra $\T_{2,\ell N}$ also acts on the jacobian $J_1(\ell N)$ of the modular curve $X_1(\ell N)$, so we can consider the subspace
\[ V_{f,\mathfrak{l}} = \bigcap_{T \in \T_{2,\ell N}} \ker \big(T - [\mu_{f,\mathfrak{l}}(T)] \big) \big\vert_{J_1(\ell N)[\ell]} \]
of the $\ell$-torsion of $J_1(\ell N)$. By section 7.9 of \cite{DS}, $V_{f,\mathfrak{l}}$ is defined over $\Q$, and by theorem 9.2 of \cite{Edi}, it has dimension $2$ as a vector space over $\F_\ell$, so that the action of $G_\Q$ on its points yields a Galois representation $\rho'_{f,\mathfrak{l}}$ into $\GLFl$ which cuts out the Galois number field $L = \overline \Q^{\ker \rho'_{f,\mathfrak{l}}} = \Q(P, \ P \in V_{f,\mathfrak{l}})$, as shown below :
\[\xymatrix @R=4pc @C=2pc { G_\Q \ar[r]^-{\rho'_{f,\mathfrak{l}}} \ar@{->>}[d] & GL(V_{f,\mathfrak{l}}) \simeq \GLFl . \\ \Gal(L/\Q) \ar@{^{(}->}[ur] & } \]

Of course, if we had $k = 2$ in the first place, there is no need to appeal to the weight-lowering theorem, and the subspace $V_{f,\mathfrak{l}}$ already exists in the $\ell$-torsion of $J_1(N)$ instead of $J_1(\ell N)$.

This representation $\rho'_{f,\mathfrak{l}}$ is unramified outside $\ell N$ (cf \cite{DS}, theorem 9.6.5). Furthermore, it follows from the Eichler-Shimura relation (cf \cite{DS}, theorem 8.7.2) that for $p \nmid \ell N$, the image of the Frobenius element $\left( \frac{L/\Q}{p} \right)$ by $\rho'_{f,\mathfrak{l}}$ has characteristic polynomial
\[ X^2 - a_p X + \varepsilon(p) p^{k-1} \in \F_\ell[X], \]
where $a_p$ and $\varepsilon(p)$ have both been reduced modulo $\mathfrak{l}$. By the Brauer-Nesbitt theorem (cf \cite{CuRe}, theorem 30.16), $\rho_{f,\mathfrak{l}}$ is therefore isomorphic to the semisimplification of $\rho'_{f,\mathfrak{l}}$, so that $\rho'_{f,\mathfrak{l}}$ is actually irreducible and thus realises $\rho_{f,\mathfrak{l}}$ indeed.

\bigskip

It would be interesting to compute explicitly these Galois representations $\rho_{f,\mathfrak{l}}$ for several reasons: first, simply for the sake of the Galois representation itself, next, because the number field $L$ will often\footnote{Under our hypotheses, the image of the representation $\rho_{f,\mathfrak{l}}$ is clealry the subgroup of $\GLFl$ made up of the matrices whose determinant is of the form $\varepsilon(n) n^{k-1}$. By ``often'', we mean that this subroup is often the whole of $\GLFl$.} be an explicit solution to the inverse Galois problem for $\GLFl$ (cf \cite{BosGal1}, \cite{BosGal2}) with controlled ramification, and even for the Gross problem, and, last but not least, because it gives a fast way of computing the $q$-expansion coefficients $a_p$ of $f$ modulo $\mathfrak{l}$. Letting $\ell$ vary, we thus obtain a Schoof-like algorithm (cf \cite{Schoof}) to compute $q$-expansions of newforms, as bounds on the coefficients $a_p$ are well-known.

Computing these representations is the goal pursued by the book \cite{EC}. The idea is to approximate $\ell$-torsion divisors representing the points of $V_{f,\mathfrak{l}}$. To compute these torsion divisors, the book \cite{EC} suggests two approaches: a probabilistic one \cite{CH13}, which creates $\ell$-torsion divisors by applying Hecke operators to random divisors on the modular curve over small finite fields, and a deterministic one \cite{CH12}, which relies on fast exponentiation to create approximations of torsion divisors on the modular curve over $\C$. However, neither of these two methods is practical at all, although their theoretical complexities are polynomial in $\ell$.

In \cite{Bos}, J. Bosman presents a practical variant of the complex method. It uses an analytic continuation method (cf for instance \cite{NumCont}) instead of fast exponentiation. To deal with the Abel-Jacobi map
\[\jmath \colon \Div^0\big(X_1(\ell N)\big)(\C) \rightarrow J_1(\ell N)(\C),\]
J. Bosman has to compute a lot of abelian integrals. This leads to precision problems as this requires summing $q$-series very close to the edge of the convergence disk, and because of the singular locus of $\jmath$, of which little is known. J. Bosman still manages to compute representations up to level $17$ and $19$, but he only gets projective Galois representations in $\PGLFl$ instead of $\GLFl$, which means he gets the coefficients $a_p \bmod \mathfrak{l}$ up to a sign only (cf for instance the table on the very first page of \cite{EC}).

To our understanding, the implementation \cite{Zeng} by J. Zeng of the probabilistic method suffers from the same limitations. J. Zeng computes polynomials defining projective non-faithful representations, but seems not to compute actual coefficients.

In this paper, we present another improved, practical and deterministic version of the complex approximations approach, which is provable in that the singular locus of $\jmath$ is no longer a problem. It has far fewer precision issues, as it computes abelian integrals only along very short paths well inside the convergence disks, and uses K. Khuri-Makdisi's algorithms \cite{Mak1} \cite{Mak2} for fast exponentiation in the jacobian. Consequently, we get approximations of torsion divisors fairly easily. This allows us to compute the full Galois representations for the prime levels $17 \leqslant \ell \leqslant 29$, which, to our knowledge, had never been done before. As a consequence, we can for instance find the signs which were missing in J. Bosman's results. We give a detailed description of our method in the following sections of this paper.

 Like J. Bosman, we limit ourselves to prime levels $\ell$ for commodity\footnote{However, our algorithm could easily be extended to general levels $N$.}. This implies that we can only use our algorithm to compute Galois representations attached to newforms of weight $2$ and level $\ell$, or to newforms of arbitrary even weight but of level $1$. Typically, we use it on the newform $\Delta$, which is of weight $12$ and level $1$. As the genus of $X_1(\ell)$ is $0$ for $\ell \leqslant 7$, we will assume $\ell \geqslant 11$ throughout this paper. The genus of $X_1(\ell)$ is then $g = \frac{(\ell-5)(\ell-7)}{24}$.

In the next section, we present a quick review of our algorithm. Then, in section $3$, we give a detailed description of the key steps. Finally, in the last section, we present actual computations of Galois representations and of coefficients of newforms, and we give complexity estimates.

\newpage

\section{Outline of the algorithm}

Our first task consists in computing the period lattice $\Lambda$ of $X_1(\ell)$, which we do by integrating cuspforms along modular symbols. Using our knowledge of the action of the Hecke algebra on modular symbols, we then deduce an analytic representation of the $\ell$-torsion subspace $V_{f,\mathfrak{l}} \subset J_1(\ell)(\C) = \C^g / \Lambda$. Next, we find a way to invert the Abel-Jacobi map $\jmath$, so that we may, for each $x \in J_1(\ell)(\C)$, find a null-degree divisor $D_x$ on $X_1(\ell)$ such that $\jmath(D_x) = x$, and especially so for two $\ell$-torsion divisor classes $x_1$ and $x_2$ forming a basis of the two-dimensional $\F_\ell$-subspace $V_{f,\mathfrak{l}}$.

We first compute a high-precision floating point approximation of the period lattice $\Lambda$ by computing a $\Z$-basis of the singular homology $H_1\big( X_1(\ell)(\C),\Z \big)$ made up of modular symbols (cf \cite{Stein} or \cite{Cremona}), along which we integrate term-by-term the $q$-expansions of a basis $(f_i)_{1\leqslant i \leqslant g}$ of cuspforms of weight $2$ In order to get a very accurate result, this requires $q$-expanding the $f_i$ to high precision, which we show how to do quickly \hyperlink{Detail_qexp}{below}. Then, by computing the Hecke action on $J_1(\ell)[\ell]$, we can express our two divisor classes $x_1$ and $x_2$ as points of $\frac1\ell \Lambda / \Lambda \subset \C^g / \Lambda$.

Let $\widetilde{x_1}$ be a lift of $x_1$ to $\C^g$. We next pick $g$ points $(P_j)_{1\leqslant j \leqslant g}$ on $X_1(\ell)$, and, using Newton iteration, we compute another $g$ points $(P'_j)_{1\leqslant j \leqslant g}$ with $P'_j$ close to $P_j$ such that
\[ \sum_{j=1}^g \left( \int_{P_j}^{P'_j} f_i(\tau) d\tau \right)_{1 \leqslant j \leqslant g} = \frac{\widetilde{x_1}}{2^m}, \]
 where $m \in \N$ is large enough for Newton iteration to converge, and the integrals are taken along the short paths joining $P_j$ to $P'_j$. Thus, we get the divisor
\[ D_1^{(m)} = \sum_{j=1}^g ( P'_j-P_j) \]
which satisfies $2^m \big[D_1^{(m)}\big] = x_1$. Then, using K. Khuri-Makdisi's algorithms \cite{Mak1} \cite{Mak2} to compute in the jacobian $J_1(\ell)$, we double $m$ times the divisor class of $D_1^{(m)}$, which yields an $\ell$-torsion divisor $D_1$ representing $x_1$. We apply the same process to get another $\ell$-torsion divisor $D_2$ representing $x_2$.

This way, we find $\ell$-torsion divisors using only integrals along short paths which are well inside the convergence disks. Consequently we have far fewer precision problems than with J. Bosman's method \cite{Bos}.

We thus now have two $\ell$-torsion divisors $D_1$ and $D_2$ whose images by the Abel-Jacobi map form a basis of the $\ell$-torsion subspace $V_{f,\mathfrak{l}}$. We then compute all the reduced divisors
\[ D_{a,b} \sim a D_1 + b D_2, \quad a,b \in \F_\ell, \]
yielding a collection of $\ell^2$ reduced divisors corresponding to the $\ell^2$ points of $V_{f,\mathfrak{l}}$, and evaluate them by a well-chosen Galois-equivariant map $\alpha \colon V_{f,\mathfrak{l}} \longrightarrow \overline \Q$. The polynomial
\[ F(X) = \prod_{\substack{a,b \in \F_\ell \\ (a,b) \neq (0,0)}} \big( X - \alpha(D_{a,b}) \big) \]
then lies in $\Q[X]$; we can recognise its coefficients using continued fractions. This polynomial encodes the Galois representation we are attempting to compute, in that its splitting field $L$ over $\Q$ is the number field cut out by the representation $\rho_{f,\mathfrak{l}}$, and $\Gal(L/\Q)$ acts on its roots $ \varphi(D_{a,b})$ just like $GL_2(\F_{\ell})$ acts on $(a,b) \in \F_\ell^2$.

Our final task is to describe the image of Frobenius elements by this representation. For this, we adapt T. and V. Dokchitser's work \cite{Dok} to get resolvents
\[ \Gamma_C(X) \in \Q[X], \quad C \text{ similarity class of } \GLFl \]
such that
\[ \rho_{f,\mathfrak{l}} \left( \left( \frac{L/\Q}{p} \right) \right) \in C \ \Longleftrightarrow  \ \Gamma_C \big( \Tr_{A_p/\F_p}a^p \, h(a) \big) = 0 \bmod p, \]
where $A_p = \F_p[X] / \big(F(X)\big)$, $a$ denotes the class of $X$ in $A_p$, and $h$ is a polynomial (cf \cite{Dok} or section \ref{Detail_Dok}). We furthermore present a trick to reduce the amount of computations at this step.

Finally, we can now compute the coefficients $a_p$ of the $q$-expansion of $f$ modulo $\mathfrak{l}$:
\[ a_p \bmod \mathfrak{l} = \Tr \rho_{f,\mathfrak{l}} \left(\left( \frac{L/\Q}{p} \right)\right). \]

\newpage

\section{Detailed description of the steps}

We first show in subsection \ref{Detail_qexp} how to quickly compute a huge number of terms of the $q$-expansion at infinity of the cuspforms of weight $2$ and level $\ell$, and next, in \ref{Detail_periods}, how to efficiently compute the period lattice of $X_1(\ell)$ to high precision using these $q$-expansions. Then, we explain in \ref{Detail_Mak} how to use K. Khuri-Makdisi's algorithms \cite{Mak1} \cite{Mak2} on $X_1(\ell)$. Our method requires carefully chosing Eisenstein series, as explained in \ref{Detail_Eisenstein}. After this, we show in \ref{Detail_input} how to compute an $\ell$-torsion divisor. Finally, we explain in \ref{Detail_eval} how to construct a well-behaved function on the jacobian $J_1(\ell)$ and how to evaluate it at the $\ell$-torsion divisors, and we conclude by describing in \ref{Detail_Dok} an efficient way of computing the image of the Frobenius elements by the Galois representation.

\subsection{Expanding the cuspforms of weight $2$ to high precision}\label{Detail_qexp}

\hypertarget{Detail_qexp}We will need to know the $q$-expansion of the newforms of weight $2$ in order to compute the period lattice of the modular curve. Classical methods based on modular symbols (cf for instance chapter 3 of \cite{Stein}) allow us to compute a moderate number of terms of these $q$-expansions, but we will need to know the periods with very high accuracy, which requires computing a quite larger number of coefficients in these $q$-expansions, therefore using classical methods for this, though possible, would be too slow. Consequently, we present a new method to quickly compute a huge number of such coefficients. It proceeds roughly as follows :

\begin{itemize}
\item First, compute a moderate number of coefficients of the $q$-expansion of each cuspform $\omega$.
\item Then, use these coefficients to find a polynomial equation relating (a modular function depending on) $\omega$ to the modular invariant $j$, or some other modular function whose $q$-expansion is very easy to compute.
\item Finally, use Newton iteration on this equation between $q$-series to compute a huge number of coefficients of $\omega$.
\end{itemize}

Besides, all this is done modulo some prime $p$ so as to fasten the computation by avoiding intermediate coefficient swell.

\bigskip

More precisely, to compute these $q$-expansions to the precision $O(q^B)$, we first compute a generator of the Hecke algebra $\T_{2,\ell} \otimes_\Z \Q$, by picking a mild Hecke operator\footnote{In practice, it appears that in an overwhelming majority of cases, at least one of $T_2$ and $T_3$ is a generator.} and testing whether it is a generator, which is easy as it amounts to check if its eigenvalues on $S_2\big( \Gamma_1(\ell)\big)$ are all distinct.

Next, we choose the prime $p$. We can find a basis $\mathcal{B} = \bigsqcup_\varepsilon \mathcal{B}_\varepsilon$ of
\[ S_2\big( \Gamma_1(\ell)\big) = \bigoplus_{\varepsilon \text{ character} \bmod \ell}S_2(\varepsilon) \]
consisting in forms which are not necessarily eigenforms\footnote{If we used a basis of eigenforms, the common number field containing the Fourier coefficients of all these forms could be much larger.}, but which are normalised, have a nebentypus $\varepsilon$, and  $q$-expansion coefficients lying among the integers $\Z_K$ of the common cyclotomic field\footnote{The common field is indeed $\Q\left(\zeta_{(\ell-1)/2} \right)$ and not $\Q\left(\zeta_{\ell-1} \right)$, as $S_2(\varepsilon) = 0$ for odd $\varepsilon$.} $K = \Q\left(\zeta_{(\ell-1)/2} \right)$. To make it easier to reduce mod $p$ and lift back to $K$, we want $p$ to split completely in $K$. Also, $p$ should be chosen large enough for reduction mod $p$ of the coefficients to be faithful. Deligne's bounds state that if $q + \sum_{n \geqslant 2} a_n q^n$ is a newform of weight $2$, then for all $n \in \N$,
\[ \vert a_n \vert_\sigma \leqslant d(n) \sqrt n \]
for every complex embedding $\sigma$, where $d(n)$ denotes the number of positive divisors of $n$. These bounds may not apply to the forms in the bases $\mathcal{B}_\varepsilon$ as they are not eigenforms, but using our knowledge of a generator of the Hecke algebra, we can compute for each $\varepsilon$ a change of basis matrix from the basis $\mathcal{B}_\varepsilon$ to a basis of eigenforms, then deduce from Deligne's bound a bound on the complex embeddings of the $B$ first coefficients of the forms of  $\mathcal{B}_\varepsilon$, and finally, compute a bound on the coefficients of these coefficients seen as polynomials in $\zeta_{(\ell-1)/2}$. We choose $p \neq \ell$ to be the smallest rational prime greater than twice this bound and such that $p \equiv 1 \bmod (\ell-1)/2$.  Then the $(\ell-1)/2$-th cyclotomic polynomial splits completely over $\F_p$. Letting $a_i$ denote lifts to $\Z$ of its roots in $\F_p$, and $\mathfrak{p}_i = \big(p,\zeta_{(\ell-1)/2}-a_i\big)$, we have that $p$ splits completely as $\prod_i \mathfrak{p}_i$ in $K$.

Next, we compute the forms
\[ E_4 = 1 + 240 \sum_{n=1}^{+\infty} \sigma_3(n) q^n, \
E_6 = 1 - 540 \sum_{n=1}^{+\infty} \sigma_5(n) q^n, \
\mbox {and } u = \frac1j = \frac{E_4^3-E_6^2}{1728 E_6^2} \]
in $\F_p[[q]]$, as well as $dj$ in $q^{-2} \F_p[[q]] dq$, to precision $O(q^B)$.

We then can compute the $q$-expansions of the forms $\omega$ with trivial nebentypus $\varepsilon = \charf$ in $\mathcal{B}_\charf$ as follows. Note that such a form $\omega$ has $q$-coefficients in $\Z$. Consider the form $v = \frac{\omega dq}{q dj} \in \Z[[q]]$. It has weight $0$, so it is a rational function on $X_1(\ell)$, which actually descends to a rational function on $X_0(\ell)$ because $\varepsilon = \charf$. Its degree there is at most $2g_0+\ell+1$, where $g_0$ denotes the genus of $X_0(\ell)$. Indeed, its degree is at most the number of zeroes of the $1$-form $\omega \frac{dq}q$ plus the number of poles of the $1$-form $dj$. On the one hand, $\omega \frac{dq}q$ has exactly $2g_0-2$ zeroes as it is regular. On the other hand, as $dj$ has a double pole at the cusp on $X(1)$, it has a pole of order $e_c+1$ at each cusp $c$ of $X_0(\ell)$, where $e_c$ is the ramification index of $c$. Summing over the two cusps of $X_0(\ell)$, we thus see that $dj$ has $\ell+3$ poles on $X_0(\ell)$, hence the announced bound on the degree of $v$. Besides, $u$ has degree exactly $\ell+1$ on $X_0(\ell)$. Consequently, there exists an irreducible polynomial $\Phi(U,V) \in \F_p[U,V]$ of degree at most $2g_0+\ell+1$ in $U$ and exactly $\ell+1$ in $V$ such that $\Phi(u, v ) \equiv 0 \bmod p$. We compute this polynomial by linear algebra over $\F_p$ in $\F_p[[q]]$, using a moderately precise $q$-expansion of $\omega$ computed by classical algorithms. Then, by Newton iteration, we can compute $v \bmod p$, and hence $\omega \bmod p$, to the precision $O(q^B)$, and finally lift the coefficients of $\omega$ back to $\Z$.

Once this is done, we can compute the $q$-expansions of the forms $\omega$ with nontrivial nebentypus $\varepsilon$ as follows. Let $\omega_0 \in \mathcal{B}_\charf$ be one of the $g_0$ forms\footnote{Here, the method breaks down for $\ell = 13$. Indeed, this is the only case in which $g_0 = 0$ (remember we supposed $\ell \geqslant 11$), so that there is no such form in this case. So, in this special case $\ell = 13$, classical methods to expand the forms should be used instead. This is not a big problem, as this is a ``small'' case ($g$ is only $2$), so little accuracy is needed and the whole Galois representation is quite fast anyway.} with trivial nebentypus whose $q$-expansion we have just computed. Then $\frac\omega{\omega_0}$ is a rational function on $X_1(\ell)$ with nebentypus $\varepsilon$. We could thus proceed to find an equation $\Phi$ as previously by reasoning on $X_1(\ell)$ instead of $X_0(\ell)$, but this would lead to very high degrees and hence would be too slow. Instead, notice that if $o$ denotes the order of $\varepsilon$, then $v = \left( \frac\omega{\omega_0} \right)^o$ has trivial nebentypus, so descends to a function on $X_0(\ell)$, of degree at most $\frac{(2g_1-2)o}{(\ell-1)/2}$, where $g_1 = g$ denotes the genus of $X_1(\ell)$, because it has degree at most $(2g_1-2)o$ over $X_1(\ell)$. We can thus compute as previously for each $\mathfrak{p}_i$ an irreducible polynomial $\Phi(U,V) \in \F_p[U,V]$ of degree at most $\frac{(2g-2)o}{(\ell-1)/2}$ in $U$ and exactly $\ell+1$ in $V$ such that $\Phi(u, v) \equiv 0 \bmod \mathfrak{p_i}$. Then again, we use Newton iteration to compute $v  \bmod \mathfrak{p_i}$, then take the $o^\text{th}$ root to recover $\omega \bmod \mathfrak{p_i}$, and finally lift back to $K$ by Chinese remainders. 

Finally, we apply the change of basis matrices from $\mathcal{B}_\varepsilon$ to eigenforms which we computed in the beginning to get the $q$-expansions of the newforms from the $q$-expansions of the forms we have just computed.

\bigskip

This method is faster than the classical one for large $B$.

\begin{thm} For fixed prime level $\ell$, the number of bit operations required to compute the $q$-expansion of the newforms in $S_2\big(\Gamma_1(\ell)\big)$ to precision $O(q^B)$ with the algorithm described above is quasi-linear in $B$.\end{thm}

In comparison, the bit complexity of the classical algorithm based on modular symbols is at least quadratic in $B$, cf \cite{Stein}, remark 8.3.3.

\begin{proof}
First notice that for fixed level $\ell$, the change of basis matrices from the bases $\mathcal{B}_{\varepsilon}$ to eigenforms are fixed, and so is the common field $K = \Q\left(\zeta_{(\ell-1)/2} \right)$. Consequently the coefficients of $\zeta_{(\ell-1)/2}$ in the coefficients up to $q^B$ of the forms in the bases $\mathcal{B}_{\varepsilon}$ are bounded by $C \sup_{n < B} d(n) \sqrt{n}$ where $C$ is some constant which does not depend on $B$. This bound is $O(B)$ (because $d(n) = O(n^{\delta})$ for every $\delta > 0$, cf for instance \cite{HW}, theorem 315), so the smallest prime $p$ larger than twice this bound and congruent\footnote{By Dirichlet's theorem on arithmetic progressions, this congruence condition does not change the order of magnitude of $p$.} to $1 \bmod (\ell-1)/2$ is also $O(B)$, and can be found in using the sieve of Eratosthenes in $O(B \log B \log \log B)$ bit operations (cf the proof of the theorem 18.10 part ii in \cite{Gathen}). Then arithmetic operations in the residue field $\F_p$ will require $O(\log B)$ bit operations. Next, $E_4$ and $E_6$ can be computed mod $p$ to precision $O(q^B)$ in $O(B \log B \log \log B)$ bit operations using again the sieve of Eratosthenes, and $u$ and $dj$ can be computed in $O(B \log B)$ operations in $\F_p$ with fast series arithmetic. As $\ell$ is fixed, computing the short $q$-expansions and finding the equations $\Phi$, which are of fixed degree, takes fixed time. Then, one Newton iteration takes $O(B \log B)$ operations in $\F_p$ with fast arithmetic, and reaching precision $O(q^B)$ requires $O(\log B)$ such iterations. Finally, lifting back each coefficient to $K$ requires $O(\log B)$ bit operations, so lifting the forms requires $O(B \log B)$ bit operations, hence the result.
 \end{proof}

\newpage

\subsection[Computing the periods of X1(l)]{Computing the periods of $X_1(\ell)$}\label{Detail_periods}

Computing the period lattice $\Lambda$ amounts, by the Manin-Drinfeld theorem (cf \cite{Lang}, chapter IV, theorem 2.1), to compute integrals of newforms of weight $2$ along modular symbols, such as
\[ \int_{\infty}^0 f(\tau) d\tau. \]
These integrals can be computed by integrating $q$-expansions term by term. However, we have to split the integration path so that the resulting series converges. Furthermore, to increase the convergence speed, we need the path ends to lie well-inside the convergence disks.

To reduce the number of integrals we compute, we use the adjointness property of the Hecke operators with respect to the integration pairing between modular symbols and cuspforms. In general, the modular symbol $\{ \infty, 0 \}$ alone does not span the rational homology of the modular curve, even over $\T_{2,\ell} \otimes \Q$, so we introduce other modular symbols, the twisted winding elements $w_p$.

More precisely, define (cf \cite{CH6}, section 6.3), for every $p \neq \ell$ prime or $p=1$, the twisted winding element
\[ w_p = \sum_{a \bmod p} \left( \frac{a}p \right) \left\{ \infty, \frac{a}p \right\} \in \mathbb{M}_2\big(\Gamma_1(\ell)\big), \]
where $ \left( \frac{\cdot}p \right)$ denotes the Legendre symbol, which we define to be $1$ if $p=1$ for convenience. We write each basis element $\gamma_j$ of $H_1\big( X_1(\ell)(\C), \Z \big)$ as a $\T_{2,\ell} \otimes \Q$-linear combination
\[ \gamma_j = \sum_p T_{j,p} w_p, \quad T_{j,p} \in \T_{2,\ell} \otimes \Q. \]
We can compute the periods using the adjointness property of the integration pairing with respect to Hecke operators as follows:
\[ \int_{\gamma_j} f(\tau) d\tau = \int_{\sum_p T_{j,p} w_p} f(\tau) d\tau = \sum_p \int_{w_p} (T_{j,p} f)(\tau) d\tau = \sum_p \lambda_{j,p} \int_{w_p} f(\tau) d\tau, \]
where $\lambda_{j,p} \in \C$ denotes the eigenvalue of the newform $f$ for the Hecke operator $T_{j,p}$.

Consequently, all we need is to compute the integrals $\int_{w_p} f(\tau) d\tau$. Given a cuspform
\[ f = \sum_{n \geqslant 1} a_n q^n \in S_2\big( \Gamma_1(\ell) \big), \]
and $\chi$ a Dirichlet character modulo $p \neq \ell$, one defines the twisted cuspform\footnote{Indeed, twisting by a Dirichlet character whose modulus is prime to the level preserves cuspforms, though it raises the level, cf \cite{AtLi}, proposition 3.1.}
\[ f\otimes \chi = \sum_{n \geqslant 1} a_n \chi(n) q^n. \]
It is a cuspform of level $\ell p^2$. The Fricke involution $W_\ell$ transforms the form $f(\tau)$ into $\frac1{\ell \tau^2} f\left( \frac{-1}{\ell \tau} \right)$. It is useful for our purpose in that it can be used to move a $\tau$ with small imaginary part to $\frac{-1}{\ell \tau}$, which can have a much larger imaginary part. We read in \cite{CH6}, section 6.2, that if $f = q + \sum_{n \geqslant 2} a_n q^n \in S_2\big(\Gamma_1(\ell),\varepsilon\big)$ is a newform with weight $2$, level $\ell$ and character $\varepsilon$, then $W_\ell f$ is the newform with weight $2$, level $\ell$ and conjugate character $\overline \varepsilon$ defined by
\[ W_\ell f = \lambda_\ell(f) \left( q + \sum_{n \geqslant 2} \overline{a_n} q^n \right), \]
where $\lambda_\ell(f)$ is given by
\[ \lambda_\ell(f) = \left\{\begin{array}{ll} - \overline{a_\ell} & \mbox{if } \varepsilon \mbox{ is trivial,} \\ \displaystyle \frac{g(\varepsilon) \overline{a_\ell}}{\ell} & \mbox{if } \varepsilon \mbox{ is nontrivial,} \end{array} \right. \]
where $g(\cdot)$ denotes the Gauss sum of a Dirichlet character. Moreover, if $f$ is a newform of weight $2$ with character $\varepsilon$, one has the formula
\[W_{\ell p^2}(f \otimes \chi) = \frac{g(\chi)}{g(\overline \chi)} \varepsilon(p) \chi(-\ell) \cdot (W_\ell f) \otimes \overline \chi. \]
An easy computation shows that
\[ \sum_{a \bmod p} \overline\chi(a) f(\tau+a/p) = g(\overline\chi) (f \otimes \chi) (\tau). \]
This yields the formula 
\[ \int_{w_p} f(\tau) d\tau = g\bigg(\left(\frac{\cdot}{p}\right)\bigg) \int_\infty^0 \bigg(f \otimes \left(\frac{\cdot}{p}\right)\bigg) (\tau) d\tau \]
\[ = g\bigg(\left(\frac{\cdot}{p}\right)\bigg) \left(\int_\infty^{\frac{i}{p \sqrt{\ell}}} \bigg(f \otimes \left(\frac{\cdot}{p}\right)\bigg) (\tau) d\tau + \int_{\frac{i}{p \sqrt{\ell}}}^0 \bigg(f \otimes \left(\frac{\cdot}{p}\right)\bigg) (\tau) d\tau \right) \]
\[ = g\bigg(\left(\frac{\cdot}{p}\right)\bigg) \left(\int_\infty^{\frac{i}{p \sqrt{\ell}}} \bigg(f \otimes \left(\frac{\cdot}{p}\right)\bigg) (\tau) d\tau - \int_\infty^{\frac{i}{p \sqrt{\ell}}} W_{\ell p^2}\bigg(f \otimes \left(\frac{\cdot}{p}\right)\bigg) (\tau) d\tau \right) \]
\[ = \frac{g\Big(\left(\frac{\cdot}{p}\right)\Big)}{2\pi i} \sum_{n=1}^{+\infty} \left(a_n - \varepsilon(p) \left(\frac{-\ell}{p}\right) \lambda_\ell(f) \overline{a_n}\right) \frac{\left(\frac{n}{p}\right)}{n} \left(e^{-\frac{2\pi}{p \sqrt{\ell}}}\right)^n, \]
which allows us to compute the integral of a newform along a twisted winding element, and thus to finally compute the period lattice of the modular curve $X_1(\ell)$. We sum power series at $q = e^{-\frac{2\pi}{p \sqrt{\ell}}}$ for small\footnote{We have checked $p \leqslant 3$ is very often sufficient, and $p \leqslant 7$ is enough for all levels $\ell \leqslant 61$, except for $\ell = 37$ in which case we had to go up to $p = 19$.} primes $p$, which has small enough modulus to achieve fast convergence.

\subsection[Arithmetic in the jacobian J1(l)]{Arithmetic in the jacobian $J_1(\ell)$}\label{Detail_Mak}

In order to efficiently compute in the jacobian $J_1(\ell)$, we use K. Khuri-Makdisi's algorithms \cite{Mak1} \cite{Mak2}. This requires choosing an effective divisor $D_0$ of degree $d_0 \geqslant 2g+1$ such that we know how to compute the associated complete linear series
\[ V = H^0 \big( X_1(\ell), 3 D_0 \big). \]
A divisor class $x \in J_1(\ell)$ is then represented by an effective divisor $D$ of degree $d_0$ such that the class of $D-D_0$ is $x$, and $D$ is itself represented by the subspace
\[ W_D = H^0 \big( X_1(\ell), 3 D_0 - D \big) \subset V; \]
in particular $0 \in J_1(\ell)$ can be represented by
\[ W_0 = H^0\big( X_1(\ell), 2 D_0 \big) \subset V.\]

Let us first give an overview of how to achieve this. Our strategy consists in choosing $D_0 = K + c_1 + c_2 + c_3$, where $K$ is an effective canonical divisor, and the $c_i$ are cusps, so that for us $d_0 = 2g+1$ exactly. First, we compute the $(g+2)$-dimensional space
\[ V_2 = H^0\big( X_1(\ell), \Omega^1(c_1+c_2+c_3) \big). \]
This space is the direct sum of all the cusp forms of weight $2$ and of the scalar multiples of Eisenstein series $e_{1,2}$ and $e_{1,3}$ of weight $2$ vanishing at all cusps except $c_1$ and $c_2$ for $e_{1,2}$ and except $c_1$ and $c_3$ for $e_{1,3}$,
\[ V_2 = S_2\big(\Gamma_1(\ell),\C\big) \oplus \C e_{1,2} \oplus \C e_{1,3} \subset M_2\big( \Gamma_1(\ell), \C \big). \]
The point of this is that by picking a rational cusp form $f_0 \in S_2\big(\Gamma_0(\ell),\Q\big)$, we obtain an isomorphism
\[ \begin{array}{ccc} V_2 & \overset{\sim}{\longrightarrow} & H^0 \big( X_1(\ell), K+c_1+c_2+c_3 \big) \\ f & \longmapsto & \displaystyle \frac{f}{f_0} \end{array}, \]
where $K$ is the divisor of the differential $1$-form over $X_1(\ell)$ associated to the cuspform $f_0$, which is indeed an effective canonical divisor. Now by \cite{Mak1}, lemma 2.2, the map
\[ \begin{array}{ccc} V_2^{\otimes 3} & \longrightarrow & H^0 \big( X_1(\ell), 3 (K+c_1+c_2+c_3) \big) \\ f_1 \otimes f_2 \otimes f_3 & \longmapsto & \displaystyle \frac{f_1 f_2 f_3}{f_0^3} \end{array} \]
is surjective. We may thus choose $V$ to be the image of the multiplication map
\[ \begin{array}{ccc} V_2^{\otimes 3} &\longrightarrow & M_6\big( \Gamma_1(\ell), \C \big) \\ f_1 \otimes f_2 \otimes f_3 & \longmapsto & \displaystyle f_1 f_2 f_3 \end{array}. \]
In this framework, the subspace $W_0$ representing $0 \in J_1(\ell)$ is the image of the map
\[ \begin{array}{ccc} V_2^{\otimes 2} &\longrightarrow & M_6\big( \Gamma_1(\ell), \C \big) \\ f_1 \otimes f_2 & \longmapsto & \displaystyle f_1 f_2 f_0 \end{array}. \]
From now on, we will implicitly identify weight-$6$ modular form spaces with the corresponding modular function spaces obtained by dividing by $f_0^3$.

We represent the weight-$6$ forms by their $q$-expansions at each cusp\footnote{We could also have represented forms by their $q$-expansions at $\infty$ only, but we think using $q$-expansions at various cusps is better for numerical stability. Also we will later need to be able to evaluate the forms at various points of the modular curve, hence it is better to know the $q$-expansions at various places.}. The modular curve $X_0(\ell)$ has exactly two cusps, namely $\Gamma_0(\ell) \infty$ and $\Gamma_0(\ell) 0$, whereas the modular curve we're interested in, $X_1(\ell)$, has exactly $\ell-1$ cusps, half of which lie above $\Gamma_0(\ell) \infty$ while the other half lie above $\Gamma_0(\ell) 0$. We call the former cusps above $\infty$ and the latter cusps above $0$. The cusps above $0$ are all rational, whereas the cusps above $\infty$ make up a single Galois orbit. Now, the diamond operators $\langle d \rangle$, $d \in (\Z/\ell\Z)^*$, which correspond to the action of the quotient goup $\Gamma_0(\ell) / \Gamma_1(\ell) \simeq (\Z/\ell\Z)^*$, orbit the cusp $\Gamma_1(\ell) \infty$ onto the cusps above $\infty$, and the cusp $\Gamma_1(\ell) 0$ onto the cusps above $0$. Moreover, the Fricke operator $W_\ell$ swaps $\Gamma_1(\ell) \infty$ and $\Gamma_1(\ell) 0$. We know how the Fricke operator acts on newforms of weight $2$ (cf subsection \ref{Detail_periods} on the periods), and on Eisenstein series (cf next subsection \ref{Detail_Eisenstein}). Besides, all the forms we are dealing with have characters, so that the action of the diamond operators $\langle d \rangle$ on their $q$-expansions boils down to multiplying by the value of their character at $d$. Using these two kinds of operators, we thus get the $q$-expansions of the newforms and of the Eisenstein series at all cusps from their $q$-expansions at $\infty$.

\subsection{Finding the appropriate Eisenstein series}\label{Detail_Eisenstein}

We now explain how to choose the Eisenstein series $e_{1,2}$ and $e_{1,3}$. Let us first review some facts about Eisenstein series of weight $2$ in general (not necessarily prime) level $N$. From \cite{DS}, chapter 4, we know that the Eisenstein subspace of $M_2\big( \Gamma_1(N), \C \big)$ has a basis formed of the Eisenstein series
\[ G_2^{\psi,\varphi}( \tau) = \sum_{r=0}^{u-1}\sum_{s=0}^{v-1}\sum_{t=0}^{u-1} \psi(r) \overline \varphi(s) \hspace{-0.7cm} \sum_{\substack{ (c,d) \in \Z^2 \\c \equiv rv \bmod N \\ d \equiv s+tv \bmod N}} \hspace{-0.5cm} \frac1{(c \tau + d)^2}, \]
where $\psi$ and $\varphi$ are Dirichlet characters not both trivial, of the same parity, and of respective conductors $u$ and $v$ such that $uv = N$ exactly, and
\[ G_2(N\tau) - N G_2(N \tau), \quad G_2(\tau) = \sum_{c \in \Z} \sum_{\substack{d \in \Z \\ (c,d) \neq (0,0)}} \frac1{(c\tau+d)^2}. \]
We furthermore have the $q$-expansions at $\infty$
\[ E_2^{\psi,\varphi}(\tau) = - \charf_{u=1} \frac{1}{2} \sum_{a=0}^{v-1} \varphi(a) a \left( \frac{a}{v} + 1 \right) + 2 \sum_{n=1}^{+\infty} \left(\sum_{\substack{m > 0 \\ m \vert n}} \psi(n/m) \varphi(m) m \right) q^n, \]
where $E_2^{\psi,\varphi}$ is the normalisation of $G_2^{\psi,\varphi}$ defined by the relation 
\[ G_2^{\psi,\varphi} = \frac{-4 \pi^2 g(\overline \varphi)}{v^2} E_2^{\psi,\varphi}, \]
and where $g(\cdot)$ denotes the Gauss sum of a Dirichlet character, and
\[ E_2(\tau) = 1 - 24 \sum_{n=1}^{+\infty}  \left(\sum_{\substack{m > 0 \\ m \vert n}} m\right) q^n, \quad G_2 = \frac{\pi^2}3 E_2. \]
Also, $G_2^{\psi,\varphi} \in M_2\big( \Gamma_1(N), \psi \varphi \big)$ has nebentypus $\psi \varphi$, where $\psi \varphi$ is seen as a Dirichlet character modulo $N$, whereas $G_2(\tau)-N G_2(N \tau)$ has trivial nebentypus. In what follows, we will not use $G_2(\tau)-N G_2(N \tau)$ at all.

Consequently, in the case when $N = \ell$ is prime, we are left with only two cases, namely $G_2^{\chi,1}$ and $G_2^{1,\chi}$, where $\chi$ is a nontrivial even Dirichlet character modulo $\ell$. Both have nebentypus $\chi$, and $G_2^{\chi,1}$ vanishes at $\infty$ while $G_2^{1,\chi}$ does not.

We easily check the formula
\[ G_2^{\psi,\varphi}(\tau) = \sum_{(c,d) \in \Z^2} \frac{\psi(c) \overline \varphi(d)}{(vc \tau + d)^2}, \]
from which it is clear that
\[ W_N G_2^{\psi,\varphi} = \frac{u}{v} \psi(-1) G_2^{\overline \varphi, \overline \psi}, \]
and thus
\[ W_N E_2^{\psi,\phi} = \frac{g(\psi)}{g(\overline \varphi)} \frac{v}{u} \psi(-1) E_2^{\overline \phi, \overline \psi}. \]

We construct Eisenstein series $e_{1,2}$ and $e_{1,3}$ as linear combinations of the $E_2^{\chi,1}$'s and the $E_2^{1,\chi}$'s, because they have nicer $q$-expansions than their $G$-counterparts. First, we choose the cusps $c_1$, $c_2$ and $c_3$ to be\footnote{These are all distinct as $\ell \geqslant 11$.} $c_1 = \Gamma_1(\ell) 0$, $c_2 = \langle 2 \rangle c_1$, and $c_3 = \langle 3 \rangle c_1$, so that they are all defined over $\Q$\footnote{This way, as the canonical divisor $K$ is Galois-invariant since it is the divisor of $f_0 \in S_2\big(\Gamma_0(\ell),\Q\big)$, whose $q$-expansion at $0$ is thus easily proved to be rational, our divisor $D_0$ used to run K. Khuri-Makdisi algorithms will be Galois-invariant, yielding a good behaviour with respect to the Galois action.}. Next, we have from the above formulae
\[ W_\ell E_2^{\chi,1} = \frac{g(\chi)}{\ell} E_2^{1,\overline \chi} \qquad \mbox{and} \qquad W_\ell E_2^{1,\chi} = \frac{\ell}{g(\overline \chi)} E_2^{\overline \chi,1}, \]
from which we read that $E_2^{\chi,1}$ vanishes at the cusps above $\infty$ but not at the cusps above $0$, while the opposite stands true for $E_2^{1,\chi}$. Consequently we construct $e_{1,2}$ and $e_{1,3}$ as linear combinations of the $E_2^{\chi,1}$ only. Now, it follows easily from the orthogonality relations between Dirichlet characters that the Eisenstein series
\[ e_{1,2} = \sum_{\substack{\chi \text{ even} \\ \chi \neq 1}} \frac{1-\chi(2)}{\displaystyle g(\chi) \sum_{a=0}^{\ell-1} \overline \chi(a) a \left( \frac{a}{\ell} + 1 \right)} E_2^{\chi,1} \]
and
\[ e_{1,3} = \sum_{\substack{\chi \text{ even} \\ \chi \neq 1}} \frac{1-\chi(3)}{\displaystyle g(\chi) \sum_{a=0}^{\ell-1} \overline \chi(a) a \left( \frac{a}{\ell} + 1 \right)} E_2^{\chi,1} \]
meet the requirements, that is to say $e_{1,2}$ vanishes at all cusps but $c_1$ and $c_2$, and $e_{1,3}$ vanishes at all cusps but $c_1$ and $c_3$.

\subsection[{Computing an l-torsion divisor}]{Computing an $\ell$-torsion divisor}\label{Detail_input}

Recall our goal is to find null-degree divisors $D_1$ and $D_2$ representing a basis of the eigenplane $V_{f,\mathfrak{l}} \subset J_1(\ell)[\ell]$. From our knowledge of the period lattice $\Lambda$ and of a generator of the Hecke algebra $\T_{2,\ell}$, we can express the basis vectors $x_k$, $k \in \{1,2\}$ of $V_{f,\ell}$ as points in the analytic model $\C^g / \Lambda$ of the jacobian $J_1(\ell)(\C)$. Lift $x_k$ to $\widetilde{x_k} \in \C^g$. We will use Newton iteration to compute $2g$ points $P_j$ and $P'_j$, $1 \leqslant j \leqslant g$, with each $P'_j$ close to $P_j$,
such that
 \[ \sum_{j=1}^g \left( \int_{P_j}^{P'_j} f_i(\tau) d\tau \right)_{1 \leqslant j \leqslant g} = \frac{\widetilde{x_k}}{2^m}, \quad \hypertarget{jeqn}{(\star)} \]
the $2^m$ factor easing convergence. Here the integrals are taken along the short path\footnote{By this we mean that $P'_j$ lies in the same coordinate disk as $P_j$, namely the $q$-disk centered at the cusp $c_j$ (see below), and we integrate along a path inside this disk.} joining $P_j$ to $P'_j$.

\bigskip

First, pick $g$ (not necessarily distinct) cusps $c_1$, $\cdots$, $c_g$. For each of these cusps, we have an analytic map, the ``$q$-coordinate'' around $c_j$
\[ \kappa_j \colon \E \longrightarrow X_1(\ell)(\C), \]
where $\E$ stands for the open unit disk in $\C$, which maps $0$ to the cusp $c_j$ and which is a local diffeomorphism. Next, choose $g$ complex numbers $q_1$, $\cdots$, $q_g$ of small moduli, so that each point $P_j = \kappa_j(q_j)$ is close to the cusp $c_j$. Consider another vector of $g$ small complex numbers $\delta q_1$, $\cdots$, $\delta q_g$. We want to adjust this vector so that \hyperlink{jeqn}{$(\star)$} be satisfied with $P'_j = \kappa_j(q_j+\delta q_j)$. To sum up, the overall map we apply Newton iteration to is
\[ \begin{array}{ccccccc} U  \subset \E^g & \overset{\prod \kappa_j} {\longrightarrow} & X_1(\ell)^g & \longrightarrow & \Div^0 X_1(\ell) & \longrightarrow & \C^g \\ (\delta q_j)_{1 \leqslant j \leqslant g} & \longmapsto & (P'_j)_{1 \leqslant j \leqslant g} & \longmapsto & \displaystyle \sum_{j=1}^g (P'_j - P_j) & \longmapsto & \displaystyle \sum_{j=1}^g \left( \int_{P_j}^{P'_j} f_i(\tau) d\tau \right)_{1 \leqslant j \leqslant g} \end{array} \hspace{-1cm}, \]
where $U$ is a suitable neighbourhood of $0 \in \E^g$. Its differential is given by the newforms $f_i$ themselves evaluated at the $P'_j$, so this presents no difficulty.

\bigskip

Once this is done, we must double the divisor class of
\[ D_k^{(m)} = \sum_{j=1}^g (P'_j - P_j) \]
$m$ times, using K. Khuri-Makdisi's algorithms of course. However, these algorithms can only deal with divisors of the form $D - D_0$, where $D$ is an effective divisor of degree $d_0$, and $D_0$ and $d_0$ are defined in the beginning of the section \ref{Detail_Mak}. To work around this, we fix a ``padding divisor'', that is to say an effective divisor\footnote{Because we will have to evaluate $q$-series at $C$, it proves convenient to choose a divisor $C$ supported by cusps, hence the notation $C$.} $C$ of degree $d_0-g = g+1$, we input the divisors $\sum_{j=1}^g P'_j + C - D_0$ and $\sum_{j=1}^g P_j + C - D_0$ which are indeed of the form $D-D_0$, and then use K. Khuri-Makdisi's algorithm to subtract these two divisor classes. Inputting a divisor $D-D_0$ is easy : it amounts to computing the subspace $W_D = H^0 \big( X_1(\ell), 3 D_0-D \big)$ of $V = H^0 \big( X_1(\ell), 3 D_0 \big)$ consisting of functions of $V$ which vanish at $D$, which we do by evaluating the $q$-series in the basis of $V$ at the points of $D$ and doing linear algebra.

Finally, once the divisor $D_k^{(m)}$ is processed, we apply K. Khuri-Makdisi's chord algorithm $x \mapsto -2x$ on it, yielding $(-2)^m [D_k^{(m)}] = \pm x_k$. The $\pm$ sign is not a problem, because we get a basis vector for $V_{f,\mathfrak{l}}$ no matter what the sign is, and this is all we actually needed.

\subsection{Evaluating the torsion divisors}\label{Detail_eval}

We need a Galois-equivariant function $\alpha \in \Q\big(J_1(\ell)\big)$ which can be efficiently evaluated at every point $x \in V_{f,\mathfrak{l}}$ given in Khuri-Makdisi form. We then evaluate $\alpha$ at each nonzero point of $V_{f,\mathfrak{l}}$, and form the polynomial
\[ F(X) = \prod_{\substack{x \in V_{f,\mathfrak{l}} \\ x \neq 0}} \big(X - \alpha(x)\big) \in \Q[X] \]
which defines the Galois representation $\rho_{f,\mathfrak{l}}$. In order to recognise its coefficients as rational numbers, we compute the continued fraction expansion of each of them until we find a huge term. Clearly, the lower the height of $F(X)$ the better, as it requires less precision in $\C$. This means one should use an evaluation function $\alpha$ whose divisor of poles (or zeroes) belongs to an as-``small''-as-possible class in the N\'eron-Severi group of $J_1(\ell)$.

\bigskip

The classical approach, used in \cite{CH3}, \cite{CH14}, \cite{Bos} and \cite{Zeng}, consists in selecting a rational function $\xi$ on $X_1(\ell)$ defined over $\Q$ and extending it to $J_1(\ell)$ by
\[ \begin{array}{rccc} \Xi \colon & J_1(\ell) & \dashrightarrow & \C \\ & \displaystyle \sum_{i=1}^g P_i - g O & \longmapsto & \displaystyle \sum_{i=1}^g \xi(P_i) \end{array}, \]
where $O \in X_1(\ell)(\Q)$ is an origin for the Abel-Jacobi map. The divisor of the poles of $\Xi$ is
\[ (\Xi)_\infty = \sum_{Q \text{ pole of } \xi} \tau_{[Q-O]}^* \Theta, \]
where $\Theta$ is the theta divisor on $J_1(\ell)$ associated to the Abel-Jacobi map with origin $O$. We thus see that $(\Xi)_\infty$ is the sum of $\deg \xi$ translates of $\Theta$, so that $\xi$ should be chosen to have degree as low as possible. However, this degree is at least the gonality of $X_1(\ell)$, which is roughly proportional to $g$ (cf \cite{Abram}, remark 0.2).

\bigskip

We introduce a radically different method, which can be used on the jacobian of every algebraic curve $X$, the genus of which we will denote by $g$. Every point $x \in \Jac(X)$ can be written $[E_x-gO]$, where $E_x$ is an effective divisor of degree $g$ on $X$ which is generically unique, and $O \in X$ is a fixed point. Let $\Pi$ be a fixed divisor on $X$ of degree $2g$. Then the space $H^0(X,\Pi-E_x)$ is generically $1$-dimensional over $\C$, say spanned by $t_x \in \C(X)$. The divisor of $t_x$ is of the form $(t_x) = -\Pi+E_x+R_x$, where $R_x$ is a residual effective divisor of degree $g$ on $X$, which is the image of $E_x$ by the reflection
\[ \begin{array}{rccc} R_\Pi \colon & \Pic^g(X) & \longrightarrow &\Pic^g(X) \\ & \displaystyle [E] & \longmapsto &[\Pi-E] \end{array}. \]
Letting $A$ and $B$ be two points on $X$ disjoint from the support of $\Pi$, we can then define
\[ \begin{array}{rccc} \alpha \colon & \Jac(X) & \dashrightarrow & \C \\ & x & \longmapsto & \frac{t_x(A)}{t_x(B)} \end{array}. \]
This map is well-defined only on a dense Zariski subset of $\Jac(X)$ because of the genericity assumptions, and it is defined over $\Q$ if $X$, $\Pi$, $A$, $B$ and $O$ are defined over $\Q$. Moreover, it is much better-behaved than the function $\Xi$ used in the classical approach :
\begin{thm}
The divisor of poles of $\alpha$ is the sum of only two translates of the $\Theta$ divisor.
\end{thm}

\begin{proof}
$\alpha$ has a pole at $x \in \Jac(X)$ if and only if $[E_x - gO]$ or $[R_x-gO]$ are on the support of $\tau_{[B-O]}^* \Theta$. But $[R_x-gO]$ is the image of $[E_x-gO]$ by the involution $R_\Pi =  \tau_{[\Pi-2gO]} \circ [-1]$ defined above, and $[-1]^* \Theta = \tau_\mathcal{K}^* \Theta$ is the translate of $\Theta$ by the image $\mathcal{K}$ of the canonical class, cf \cite{HS}, theorem A.8.2.1.i.
\end{proof}
This is even in some sense optimal, at least for a general curve, as by the Riemann-Roch theorem for abelian varieties (cf \cite{HS}, theorem A.5.3.3), no nonconstant function on $\Jac(X)$ has a single translate\footnote{Recall that $\operatorname{NS}\big(\Jac(X)\big) = \Z \Theta$ for a generic curve $X$.} of $\Theta$ as divisor of poles.

\bigskip

In order to use this on the modular curve $X_1(\ell)$, there is a difficulty we have to overcome. In K. Khuri-Makdisi's algorithms, a divisor class $x \in J_1(\ell)$ is represented by a subspace $W_D = H^0\big( X_1(\ell), 3 D_0 -D \big) \subset V$, where $D$ is an effective divisor of degree $d_0 = 2g+1$ such that $[ D - D_0 ] = x$, but such a $D$ is far from unique\footnote{To be precise, by the Riemann-Roch theorem, there is a whole $(g+1)$-dimensional projective space of such $D$'s.}! Thus, the first thing to do is to rigidify the representation $W_D$ of $x$ into a representation which depends on $x$ only. To do this, we compute the sub-subspace
\[ W_{D,\text{red}} = H^0\big( X_1(\ell), 3 D_0 -D-C_1 \big) \subset W_D, \]
where $C_1$ is a fixed\footnote{Again, it proves convenient to choose a divisor $C_1$ supported by cusps, so that the $q$-series are effortless to evaluate, hence the notation $C_1$.} effective divisor of degree $d_1 = 2 d_0 - g$, so that $W_{D,\text{red}}$ will generically be $1$-dimensional by the Riemann-Roch theorem. Letting $s_D \in V$ be such that $s_D$ spans $W_{D,\text{red}}$ over $\C$, we know that the divisor of $s_D$ is of the form
\[ (s_D) = -3 D_0 + D + C_1 + E_D, \]
where $E_D$ is some effective divisor of degree $g$. As such, it generically sits alone in its linear equivalence class, again by the Riemann-Roch theorem. But on the other hand, if $W_D$ and $W_{D'}$ both represent the same point $x \in J_1(\ell)(\C)$, then $D \sim D'$, so that $E_D \sim E_{D'}$ as $D_0$ and $C_1$ are fixed. Consequently, we (generically) have $E_D = E_{D'}$, so that $W_D \mapsto E_D$ is the invariant we are looking for. We then use a trick \textit{\`a la} Khuri-Makdisi: we first compute
\[ s_D \cdot V = H^0\big( X_1(\ell), 6 D_0 -D-C_1-E_D \big), \] 
after which we compute
\[ H^0\big( X_1(\ell), 3 D_0 -C_1-E_D \big) = \{ v \in V \ \vert \ v W_D \subset s_D \cdot V \}, \]
all of this by linear algebra as in $\cite{Mak1}$ and $\cite{Mak2}$. Next, we fix another effective divisor\footnote{Again, same remark as above.} $C_2$ of degree $d_2 = d_0 + 1 - g$, so that the subspace $H^0\big( X_1(\ell), 3 D_0 -C_1-C_2-E_D \big)$ of the previously computed space $H^0\big( X_1(\ell), 3 D_0 -C_1-E_D \big)$ is generically one-dimensional. Letting $\Pi = 3 D_0 -C_1-C_2$, we thus have computed a function $t_D \in \C\big(X_1(\ell)\big)$ such that
\[ \C t_D = H^0\big( X_1(\ell), \Pi-E_D \big), \]
as wanted. This allows us to compute the map $\alpha$, which will be defined over $\Q$ if $C_1$, $C_2$, $A$ and $B$ are.

\bigskip

Evaluating $\alpha$ on $V_{f,\mathfrak{l}}$, we may thus hope to get a defining polynomial $F(X)$ of logarithmic height $g/2$ times less than if we had used the classical approach.

\subsection{Finding the Frobenius elements}\label{Detail_Dok}

After evaluating the torsion divisors by a suitable function, we get a polynomial $F(X) \in \Q[X]$ of degree $\ell^2-1$ whose decomposition field is the fixed field $L$ by the kernel of the Galois representation. It is thus a Galois number field, and its Galois group over $\Q$ is embedded by the representation as a subgroup of $\GLFl$. In order to completely specify the Galois representation, we would like to know the image of the Frobenius elements $\left( \frac{L/\Q}{p} \right)$ in $\GLFl$. We now explain how to compute the similarity class of the image of $\left( \frac{L/\Q}{p} \right)$ for almost all\footnote{Clearly, we have to exclude $p = \ell$, as $L$ is ramified at $\ell$. We will shortly see that we actually have to exclude finitely many other primes as well.} rational primes $p$. This can be used to get congruence relations modulo $\ell$ on the coefficients $a_p$ of the cuspform $f$, by looking at the trace of the similarity class of $\left( \frac{L/\Q}{p} \right)$.

\subsubsection{The Dockchitsers' resolvents}

For this, we specialise Tim and Vladimir Dokchitser's work \cite{Dok}. This yields the following result: denoting by $(a_i)_{1 \leqslant i < \ell^2}$ the roots of $F$ in $L$, if $h(X) \in \Z[X]$ is a polynomial with integer coefficients, then for each similarity class $C \subset \GLFl$, the resolvent polynomial
\[ \Gamma_C(X) = \prod_{\sigma \in C} \left( X - \sum_{i=1}^n h(a_i) \sigma(a_i) \right) \]
lies in $\Q[X]$. Moreover, if $p$ is a rational prime which divides none of the denominators of the coefficients of $F$, so that the polynomials $\Gamma_C$ are $p$-integral, which does not divide the discriminant of $F$, and such that the $\Gamma_C$'s are pairwise coprime modulo $p$, then the image by the Galois representation of the Frobenius element $\left( \frac{L/\Q}{p} \right)$ lies in the similarity class $C$ if and only if
\[ \Gamma_C \left( \Tr_{\frac{\F_p[X]}{F(X)}/ \F_p} h(a) a^p \right) = 0, \]
where $a$ denotes the class of $X$ in the quotient algebra $\F_p[X]/\big(F(X)\big)$. Furthermore, the polynomials $\Gamma_C$ are pairwise coprime over $\Q$ for a generic choice of $h(X)$ amongst the polynomials of degree at most $\ell^2-2$ with coefficients in $\Z$.

We first start by computing the roots $a_i$ to a very high precision in $\C$ using Newton iteration (note we already know them to a mildly high precision). Then, we compute complex approximations of the resolvents $\Gamma_C$ by enumerating matrices in the similarity classes of $\GLFl$. Finally, we recognise the coefficients of the resolvents as rationals, using our knowledge of an \textit{a priori} multiple of their denominators, namely a common denominator for the coefficients of $F$ to the cardinality of $C$ times one plus the degree of $h$. In practice, the choice $h(X) = X^2$ has always worked, in that the resulting resolvents $\Gamma_C$ we have computed have always been pairwise coprime over $\Q$, and actually, in all the computations we have run, they have even always remained coprime modulo $p$ as long as $p$ was reasonably large\footnote{As the primary goal of our computations is to find the coefficients $a_p$ of the $q$-expansion of $f$ modulo $\ell$, the only case we are really interested in is the case in which $p$ is extremely large, as naive methods compute $a_p$ for small $p$ in almost no time anyway.}, say at least $10$ decimal digits.

Once the resolvents are computed, it is easy to compute what Frobenius elements $\left( \frac{L/\Q}{p} \right)$ are similar to, and hence to deduce the coefficients of the cuspform $f$ modulo $\mathfrak{l}$.

\subsubsection{The quotient representation trick}

\hypertarget{Quotrep_trick} Unfortunately, these computations, although simple, can be rather slow because of the need for very high precision\footnote{For instance, in level $\ell = 29$, about 5 million decimal digits after the decimal point are required to compute the resolvents.} in $\C$. However, a simple trick allows us to sharply reduce the amount of computations needed. Indeed, we have not yet used the fact that we know in advance what the determinant of the image of the Frobenius element $\left( \frac{L/\Q}{p} \right)$ is, namely $\varepsilon(p) p^{k-1}$, where $k$ and $\varepsilon$ denote respectively the weight and the nebentypus of the newform $f$.

The idea is then to compute a \emph{quotient} representation, that is to say the representation $\rho_{f,\mathfrak{l}}$ composed with the projection map from $\GLFl$ onto one its quotient groups. The coarser the chosen quotient group, the smaller the computation, so we should use a quotient just fine enough to be able to lift correctly an element back to $\GLFl$ based on the knowledge of its determinant. Thus $\PGLFl$ for instance is slightly too coarse, because the knowledge of the image of a matrix in $\PGLFl$ and of its determinant only determines this matrix up to sign\footnote{This is the reason why J. Bosman, for computing only the projective Galois representation, determined the coefficients $a_p$ of $f$ only up to sign.}. This example clearly hints at the quotient group
\[ \widetilde{\GLFl} = \GLFl / S, \]
where $S$ is the largest subgroup\footnote{This subgroup $S$ is the subgroup made up of the elements of odd order in $\F_\ell^*$, that is to say, the $2'$-subgroup of $\F_\ell^*$.} of $\F_\ell^*$ not containing $-1$, which we choose.

Computing the associated quotient Galois representation
\[\xymatrix @R=4pc @C=2pc { G_\Q \ar[r]^-{\rho_{f,\mathfrak{l}}} & \GLFl  \ar@{->>}[r] & \widetilde{\GLFl} } \]
then amounts to describing the Galois action on
\[ \widetilde{V_{f,\mathfrak{l}}} = V_{f,\mathfrak{l}} / S. \]
We thus first begin by computing the polynomial $\widetilde F (X) \in \Q[X]$ defining $\widetilde{V_{f,\mathfrak{l}}}$ by tracing the roots $\alpha(x)$, $x \in V_{f,\mathfrak{l}}$ of $F(X)$ along their orbits under $S$ :
\[ \widetilde F (X) = \prod_{\substack{Sx \in \widetilde{V_{f,\mathfrak{l}}} \\ x \neq 0}} \left(X -  \sum_{s  \in S} \alpha(sx)\right). \]
This new polynomial has the same height as the original $F(X)$, but its degree is $\vert S \vert$ times smaller.

We must then compute the resolvents $\Gamma_{\widetilde C}(X)$ for each conjugacy class $\widetilde C$ of $ \widetilde{\GLFl}$. As the subgroup $S$ of $\GLFl$ is central, these conjugacy classes are easy to describe.

\begin{lem}
Let $\xymatrix @R=4pc @C=2pc {\pi \colon \GLFl \ar@{->>}[r] & \widetilde{\GLFl}}$ denote the projection map, let $\widetilde g \in  \widetilde{\GLFl}$, and let $g \in \GLFl$ such that $\pi(g) = \widetilde g$. Then $\pi$ induces a bijection
\[ \begin{array}{cccc} \pi_g \colon & \text{Conjugacy class of } g & \overset{\sim}{\longrightarrow} & \text{Conjugacy class of } \widetilde{g} \\ & h g h^{-1} & \longmapsto &\pi(h g h^{-1}). \end{array} \]
\end{lem}

\begin{proof}
It is clear that the image of the conjugacy class of $g$ by $\pi$ is exactly the conjugacy class of $\widetilde{g}$, so that $\pi_g$ is well-defined and surjective. To show that $\pi_g$ is also injective, let $h_1$, $h_2 \in \GLFl$ such that $\pi(h_1 g h_1^{-1}) =  \pi(h_2 g h_2^{-1})$, that is to say such that $h_1 g h_1^{-1} = s h_2 g h_2^{-1}$ for some $s \in S$. We must prove that $h_1 g h_1^{-1} = h_2 g h_2^{-1}$. By taking determinants, we see that $\det s = 1$. As $s$ is scalar, this implies $s = \pm 1$. Since $-1 \not \in S$, we conclude that $s = 1$, and therefore $h_1 g h_1^{-1} = h_2 g h_2^{-1}$.
\end{proof}

A resolvent $\Gamma_{\widetilde C}(X)$ has therefore exactly the same degree as (each of) the corresponding $\Gamma_{C}(X)$, so we must still use the same very high precision in $\C$ to compute it. However, we have now $\vert S \vert$ times less such resolvents to compute. Furthermore, the roots $\sum_{i=1}^n h(a_i) \sigma(a_i)$ of these resolvents actually take $\vert S \vert^2$ less time to compute, as they are defined by sums $\vert S \vert$ times shorter and there are $\vert S \vert$ times less of them.

Using these resolvents $\Gamma_{\widetilde C}(X)$, we can then compute the conjugacy class of the image of the Frobenius element $\left( \frac{L/\Q}{p} \right)$ in $\widetilde{\GLFl}$ as above, and, since $-1 \not \in S$, we can deduce the similarity class of the image of the Frobenius element in $\GLFl$ using our knowledge of its determinant. Consequently, with this trick, we can still compute the full, non-quotient representation $\rho_{f,\mathfrak{l}}$, and we have saved a factor $\vert S \vert^2$ in the computation of the roots of the resolvent, and a factor $\vert S \vert$ in their expansion and in the identification of their coefficients as rational numbers. Since
\[ \vert S \vert = \frac{\ell-1}{2^{\ord_2(\ell-1)}}, \]
this prevents this final step of the Galois representation computation from being the slowest one, cf the \hyperlink{Complexity}{complexity section} after the results.

\newpage

\section{Results}

As the above algorithms compute the full Galois representation, we get results which are more complete than the ones from \cite{Bos}. For instance, picking $\ell = 19$ (which corresponds to genus $g = 7$), we can compute the Galois representation $\rho_{\Delta,19}$ modulo $19$ associated to the newform
\[ f = \Delta = q \prod_{n=1}^{+\infty} (1-q^n)^{24} = \sum_{n=1}^{+\infty} \tau(n) q^n \]
of weight $12$, find the similarity class in $GL_2(\F_{19})$ of the images of Frobenius elements, and hence find the signs which were missing in the table on the very first page of \cite{EC} :

\begin{enumerate}

\item[$\bullet$] The image of the Frobenius at $p = 10^{1000}+1357$ is similar to $\left[ \begin{array}{cc} 17 & 1 \\ 0 & 17  \end{array} \right]$, therefore $\tau(10^{1000}+1357) \equiv -4 \bmod 19$,

\item[$\bullet$] The image of the Frobenius at $p = 10^{1000}+7383$ is similar to $\left[ \begin{array}{cc} 1 & 1 \\ 0 & 1  \end{array} \right]$, therefore $\tau(10^{1000}+7383) \equiv +2 \bmod 19$,

\item[$\bullet$] The image of the Frobenius at $p = 10^{1000}+21567$ is similar to $\left[ \begin{array}{cc} 11 & 1 \\ 0 & 11  \end{array} \right]$, therefore $\tau(10^{1000}+21567) \equiv +3 \bmod 19$,

\item[$\bullet$] The image of the Frobenius at $p = 10^{1000}+27057$ is similar to $\left[ \begin{array}{cc} 10 & 0 \\ 0 & 9  \end{array} \right]$, therefore $\tau(10^{1000}+27057) \equiv 0 \bmod 19$,

\item[$\bullet$] The image of the Frobenius at $p = 10^{1000}+46227$ is similar to $\left[ \begin{array}{cc} 0 & 14 \\ 1 & 0  \end{array} \right]$, therefore $\tau(10^{1000}+46227) \equiv 0 \bmod 19$,

\item[$\bullet$] The image of the Frobenius at $p = 10^{1000}+57867$ is similar to $\left[ \begin{array}{cc} 17 & 0 \\ 0 & 2  \end{array} \right]$, therefore $\tau(10^{1000}+57867) \equiv 0 \bmod 19$,

\item[$\bullet$] The image of the Frobenius at $p = 10^{1000}+64749$ is similar to $\left[ \begin{array}{cc} 13 & 1 \\ 0 & 13  \end{array} \right]$, therefore $\tau(10^{1000}+64749) \equiv +7 \bmod 19$,

\item[$\bullet$] The image of the Frobenius at $p = 10^{1000}+68367$ is similar to $\left[ \begin{array}{cc} 14 & 0 \\ 0 & 5  \end{array} \right]$, therefore $\tau(10^{1000}+68367) \equiv 0 \bmod 19$,

\item[$\bullet$] The image of the Frobenius at $p = 10^{1000}+78199$ is similar to $\left[ \begin{array}{cc} 15 & 1 \\ 0 & 15  \end{array} \right]$, therefore $\tau(10^{1000}+78199) \equiv -8 \bmod 19$,

\item[$\bullet$] The image of the Frobenius at $p = 10^{1000}+128647$ is similar to $\left[ \begin{array}{cc} 0 & 8 \\ 1 & 0  \end{array} \right]$, therefore $\tau(10^{1000}+128647) \equiv 0 \bmod 19$.
\end{enumerate}

The surprising number of occurences of non-semi-simple matrices --- by the Chebotarev theorem, non-semi-simple matrices should occur with density about $1/\ell$ only --- and of $\tau(p) \equiv 0 \bmod 19$ above can be explained by the fact that J. Bosman purposedly chose special values of $p$ (cf \cite{CH7}, section 7.4). For instance, for the few other first primes above $10^{1000}$, we have computed the following:

\[ 
\begin{array}{|c|c|c|}
\hline
p & \text{Similarity class of } \vphantom{\bigg\vert} \left( \frac{L/\Q}{p} \right) & \tau(p) \bmod 19 \\
\hline
10^{1000}+ 453 & \left[ \begin{array}{cc} 15 & 0 \\ 0 & 10 \end{array} \right] \vphantom{\Bigg\vert} & 6 \\
\hline
10^{1000}+2713 & \left[ \begin{array}{cc} 11 & 0 \\ 0 & 4 \end{array} \right] \vphantom{\Bigg\vert} & 15 \\
\hline
10^{1000}+4351 & \left[ \begin{array}{cc} 6 & 0 \\ 0 & 4 \end{array} \right] \vphantom{\Bigg\vert} & 10 \\
\hline
10^{1000}+5733 & \left[ \begin{array}{cc} 16 & 0 \\ 0 & 1 \end{array} \right] \vphantom{\Bigg\vert} & 17 \\
\hline
10^{1000}+10401 & \left[ \begin{array}{cc} 0 & 15 \\ 1 & 8 \end{array} \right] \vphantom{\Bigg\vert} & 8 \\
\hline
10^{1000}+11979 & \left[ \begin{array}{cc} 16 & 0 \\ 0 & 13 \end{array} \right] \vphantom{\Bigg\vert} & 10 \\
\hline
10^{1000}+17557 & \left[ \begin{array}{cc} 0 & 5 \\ 1 & 11 \end{array} \right] \vphantom{\Bigg\vert} & 11 \\
\hline
10^{1000}+22273 & \left[ \begin{array}{cc} 13 & 0 \\ 0 & 1 \end{array} \right] \vphantom{\Bigg\vert} & 14 \\
\hline
10^{1000}+24493 & \left[ \begin{array}{cc} 14 & 0 \\ 0 & 10 \end{array} \right] \vphantom{\Bigg\vert} & 5 \\
\hline
10^{1000}+25947 & \left[ \begin{array}{cc} 0 & 4 \\ 1 & 5 \end{array} \right] \vphantom{\Bigg\vert} & 5 \\
\hline
10^{1000}+29737 & \left[ \begin{array}{cc} 0 & 12 \\ 1 & 7 \end{array} \right] \vphantom{\Bigg\vert} & 7 \\
\hline
10^{1000}+41599 & \left[ \begin{array}{cc} 18 & 0 \\ 0 & 15 \end{array} \right] \vphantom{\Bigg\vert} & 14 \\
\hline
\end{array} \]
This agrees with the Chebotarev theorem.

\newpage

We have implemented the above algorithms in \cite{SAGE}, version $5.3$, and have run them on PlaFRIM, the Bordeaux 1 university computing cluster. For $\ell = 19$, the computation times were as follows: computing the $q$-expansion of the cuspforms and the Eisenstein series $\big($to $O(q^{5000})\big)$, the period lattice, and finally initialising K. Khuri-Makdisi's algorithms by computing the spaces $V$ and $W_0$ took 11 minutes, computing the two $19$-torsion divisors took 24 minutes (each, but they were processed in parallel), and computing all the points in the $\F_{19}$-plane spanned by them took about 40 minutes. We found a polynomial $F(X) \in \Q[X]$ defining the representation, of degree $360 = 19^2-1$ and with a common denominator of $142$ decimal digits, and finally, computing the resolvents $\Gamma_{\widetilde C}(X)$ took a little less than $20$ minutes thanks to the \hyperlink{Quotrep_trick}{quotient representation trick} and to massive parallelisation, after which deducing the similarity classes of the image of a Frobenius element at $p \approx 10^{1000}$ takes about $30$ minutes. Overall, the whole computation thus lasted about $2$ hours. We led the computation of the defining polynomial $F(X)$ with a precision of $1500$ bits in $\C$, and we used a precision of $600$ kbits to compute the resolvents $\Gamma_{\widetilde C}(X)$.

\subsection*{Level $\ell = 23$ (genus $g = 12$)}

The Galois representation modulo $\ell = 23$ associated to $\Delta$ degenerates and has dihedral image $\mathfrak{S}_3$. This phenomenon is related to Ramanujan-type congruences for $\tau(n) \bmod 23$, cf \cite{CH1}, top of page 5. The prime $\ell = 23$ is indeed one of the finitely many primes we have to exclude, cf the \hyperlink{BadPrimes}{beginning of the introduction}. Consequently, we computed the representation associated to the newform $E_4 \Delta$ of level $1$ and weight $16$ instead. We got a defining polynomial $F(X)$ of degree $528 = 23^2-1$ with a common denominator of $508$ decimal digits. Computing the period lattice took a little less than $2$ hours, computing each of the two $23$-torsion divisors took $5$ hours and a half, and computing the $\F_{23}$-plane spanned by them took a little more than $11$ hours. Overall, getting the polynomial $F(X)$ took less than 20 hours.

\newpage

\subsection*{Level $\ell = 29$ (genus $g = 22$)}

We have also computed the polynomial $F(X)$ for the Galois representation modulo $\ell = 29$ associated to $\Delta$, which took about $10$ days. This polynomial has degree $840 = 29^2-1$, and a common denominator of $1793$ decimal digits. Computing the periods took a little more than $6$ hours, computing each of the two $29$-torsion divisors took $120$ hours, and computing the $\F_{29}$-plane spanned by them took about $100$ hours.

Then, thanks to the \hyperlink{Quotrep_trick}{quotient representation trick}, computing the resolvents $\Gamma_{\widetilde C}(X)$ took about $60$ hours, and finally, deducing the image of the Frobenius at the same primes $p \approx 10^{1000}$ as in level $19$ took $2$ hours. Overall, the whole computation thus took less than two weeks.

We led the computation of the defining polynomial $F(X)$ with a precision of $15$ kbits in $\C$, and we used a precision of $18$ Mbits to compute the resolvents $\Gamma_{\widetilde C}(X)$.

\bigskip

Our results are the following :

\[ 
\begin{array}{|c|c|c|}
\hline
p & \text{Similarity class of } \vphantom{\bigg\vert} \left( \frac{L/\Q}{p} \right) & \tau(p) \bmod 29 \\
\hline
10^{1000}+ 453 & \left[ \begin{array}{cc} 0 & 5 \\ 1 & 21 \end{array} \right] \vphantom{\Bigg\vert} & 21 \\
\hline
10^{1000}+1357 & \left[ \begin{array}{cc} 0 & 28 \\ 1 & 8 \end{array} \right] \vphantom{\Bigg\vert} & 8 \\
\hline
10^{1000}+2713 & \left[ \begin{array}{cc} 0 & 9 \\ 1 & 11 \end{array} \right] \vphantom{\Bigg\vert} & 11 \\
\hline
10^{1000}+4351 & \left[ \begin{array}{cc} 0 & 26 \\ 1 & 0 \end{array} \right] \vphantom{\Bigg\vert} & 0 \\
\hline
10^{1000}+5733 & \left[ \begin{array}{cc} 20 & 0 \\ 0 & 2 \end{array} \right] \vphantom{\Bigg\vert} & 22 \\
\hline
10^{1000}+7383 & \left[ \begin{array}{cc} 19 & 0 \\ 0 & 10 \end{array} \right] \vphantom{\Bigg\vert} & 0 \\
\hline
10^{1000}+10401 & \left[ \begin{array}{cc} 7 & 0 \\ 0 & 2 \end{array} \right] \vphantom{\Bigg\vert} & 9 \\
\hline
10^{1000}+11979 & \left[ \begin{array}{cc} 0 & 7 \\ 1 & 7 \end{array} \right] \vphantom{\Bigg\vert} & 7 \\
\hline
10^{1000}+17557 & \left[ \begin{array}{cc} 0 & 2 \\ 1 & 0 \end{array} \right] \vphantom{\Bigg\vert} & 0 \\
\hline
10^{1000}+21567 & \left[ \begin{array}{cc} 23 & 0 \\ 0 & 3 \end{array} \right] \vphantom{\Bigg\vert} & 26 \\
\hline
10^{1000}+22273 & \left[ \begin{array}{cc} 0 & 26 \\ 1 & 14 \end{array} \right] \vphantom{\Bigg\vert} & 14 \\
\hline
\end{array}
\]

\newpage

\[ 
\begin{array}{|c|c|c|}
\hline
p & \text{Similarity class of } \vphantom{\bigg\vert} \left( \frac{L/\Q}{p} \right) & \tau(p) \bmod 29 \\
\hline
10^{1000}+24493 & \left[ \begin{array}{cc} 0 & 13 \\ 1 & 4 \end{array} \right] \vphantom{\Bigg\vert} & 4 \\
\hline
10^{1000}+25947 & \left[ \begin{array}{cc} 27 & 0 \\ 0 & 15 \end{array} \right] \vphantom{\Bigg\vert} & 13 \\
\hline
10^{1000}+27057 & \left[ \begin{array}{cc} 0 & 6 \\ 1 & 7 \end{array} \right] \vphantom{\Bigg\vert} & 7 \\
\hline
10^{1000}+29737 & \left[ \begin{array}{cc} 23 & 0 \\ 0 & 10 \end{array} \right] \vphantom{\Bigg\vert} & 4 \\
\hline
10^{1000}+41599 & \left[ \begin{array}{cc} 13 & 0 \\ 0 & 5 \end{array} \right] \vphantom{\Bigg\vert} & 18 \\
\hline
10^{1000}+46227 & \left[ \begin{array}{cc} 0 & 26 \\ 1 & 22 \end{array} \right] \vphantom{\Bigg\vert} & 22 \\
\hline
10^{1000}+57867 & \left[ \begin{array}{cc} 13 & 0 \\ 0 & 11 \end{array} \right] \vphantom{\Bigg\vert} & 24 \\
\hline
10^{1000}+64749 & \left[ \begin{array}{cc} 0 & 1 \\ 1 & 15 \end{array} \right] \vphantom{\Bigg\vert} & 15 \\
\hline
10^{1000}+68367 & \left[ \begin{array}{cc} 0 & 3 \\ 1 & 3 \end{array} \right] \vphantom{\Bigg\vert} & 3 \\
\hline
10^{1000}+78199 & \left[ \begin{array}{cc} 17 & 0 \\ 0 & 14 \end{array} \right] \vphantom{\Bigg\vert} & 2 \\
\hline
10^{1000}+128647 & \left[ \begin{array}{cc} 0 & 27 \\ 1 & 24 \end{array} \right] \vphantom{\Bigg\vert} & 24 \\
\hline
\end{array} \]

\bigskip

Putting together the results in level $\ell = 19$ and $\ell = 29$, we see that $\tau(p) \neq 0$ for each value of $p$ we have tested. This agrees with Lehmer's conjecture.

\newpage

\subsection*{Complexity estimates}

Clearly, the most time-consuming part of the computation of the polynomial $F(X) \in \Q[X]$ defining the representation is the arithmetic in the jacobian $J_1(\ell)$. K. Khuri-Makdisi's algorithms rely on linear algebra on matrices of size $O(g) \times O(g)$; as $g = O(\ell^2)$, and we have $O(\ell^2)$ points to compute in the jacobian, this implies a complexity of $O(\ell^8)$ operations in $\C$ to compute the Galois representation. Let $H$ be the logarithm of the common denominator of $F(X)$, so that computing $F(X)$ with our method requires a precision of $O(H)$ bits in $\C$. Then the complexity of our method to find $F(X)$ is $\widetilde O(\ell^8 H)$ bit operations. The experiments we have run seem to indicate that $H$ is $O(\ell^3)$, but we do not try to refine this estimate, because we do not know a proven sharp bound on $H$.

\hypertarget{Complexity} Next, if we do not use the \hyperlink{Quotrep_trick}{quotient representation trick}, computing a root $\sum_{i = 1}^n h(a_i) \sigma(a_i)$ of a Dokchitsers' resolvent $\Gamma_C(X)$ requires $O(n) = O(\ell^2)$ operations in $\C$. As there is one such root for each $\sigma \in \GLFl$, computing all these roots requires $O(\ell^6)$ operations in $\C$. Then, computing a resolvent $\Gamma_C(X)$ from its roots requires $\widetilde O\big(\deg \Gamma_C(X)\big) = \widetilde O(\ell^2)$ operations in $\C$ using a fast Fourier transform. As there are $O(\ell^2)$ similarity classes in $\GLFl$, we see that computing all the resolvents $\Gamma_C(X)$ from their roots requires $\widetilde O(\ell^4)$ operations in $\C$. Thus computing all the resolvents overall requires $O(\ell^6)$ operations in $\C$, the slow part being the computation of their roots. The precision in $\C$ we have to work at for this is $O(\ell^2 H)$, so that the total complexity of the computation of the resolvents $\Gamma_C(X)$ is $\widetilde O(\ell^{8} H)$ bit operations, which is the same as the rest of the computation.

However, with the  \hyperlink{Quotrep_trick}{quotient representation trick}, computing the resolvent roots $\sum_{i = 1}^n h(a_i) \sigma(a_i)$ requires only $O(\ell^6/ \vert S \vert^2) = O(\ell^4 \ell_2^2)$ operations in $\C$, where $\ell_2 = 2^{\ord_2(\ell-1)}$ is the $2$-primary part of $\ell$, and then computing the resolvents $\Gamma_{\widetilde{C}}(X)$ from these roots takes only $\widetilde O(\ell^4 / \vert S \vert) = \widetilde O(\ell^3 \ell_2)$ operations in $\C$. Therefore, in the good cases, that is if $\ell_2$ is bounded\footnote{For instance, this is the case if we restrict to the primes $\ell \equiv -1 \bmod 4$. Note that such $\ell$ are enough to compute the coefficients $a_p$ by Chinese remainders from the $a_p \bmod \mathfrak{l}$ for various $\ell$ without worsening the complexity of this method.}, the use of the quotient representation trick allows us to bring the complexity of the computation of the resolvents $\Gamma_{\widetilde{C}}(X)$ down to $\widetilde O(\ell^6 H)$ bit operations, making it $O(\ell^2)$ times faster than the rest of the computation. On the other hand, this trick is unfortunately totally inefficient in the worst cases $\ell = 2^\lambda + 1$, $\lambda \in \N$.

\end{document}